\documentclass{amsart}
\usepackage{graphicx}
\usepackage{hyperref}
\usepackage{fancyhdr}
\usepackage[mathscr]{euscript}
\usepackage{xurl}
\usepackage{tikz}
\newtheorem{theorem}{Theorem}[section]
\newtheorem{lemma}[theorem]{Lemma}

\theoremstyle{definition}
\newtheorem{definition}[theorem]{Definition}

\newenvironment{proofof}[1]{\par\noindent{\em Proof of #1.}}%
                        {\hspace*{\fill}\nobreak$\Box$\par\medskip}

\author{Lycka Drakengren}

\newcommand{\mylongtitle}[1]{%
  \ifodd\value{page}%
    \protect\parbox{0.97\linewidth}{#1}\hfill%
  \else%
    \hfill\protect\parbox{0.97\linewidth}{#1}%
  \fi%
}

\begin{document}

\title[\mylongtitle{Fundamental Groups of 2-Complexes with \\ Nonpositive Planar Sectional Curvature}]{Fundamental Groups of 2-Complexes with \\ Nonpositive Planar Sectional Curvature}

\maketitle
\begin{abstract}
We show that the finite simply connected 2-complexes of nonpositive planar sectional curvature are collapsible. Moreover, we show that each finite connected 2-complex with negative planar sectional curvature and fundamental group $\mathbb{Z}$ can be collapsed to a 1-dimensional cycle.
\end{abstract}

\vspace{8pt}
\section{Introduction}
Cell complexes with nonpositively curved metrics have been of great interest to many researchers in geometric group theory. As a consequence, several results are known concerning the fundamental groups of nonpositively curved 2-complexes (see e.g. \cite{OP}, \cite{BB}).

For Euclidean complexes, the condition of being nonpositively curved can be checked using Gromov's link condition. For 2-complexes, Gromov's link condition can be stated in terms of angles: A Euclidean 2-complex is nonpositively curved if and only if each cycle in the vertex links has length $\geq 2\pi$.

We therefore find it natural to consider 2-complexes with positive angles assigned to the corners of each face, without imposing a metric. These are referred to as \textit{angled 2-complexes}.

The notion of \textit{nonpositive planar sectional curvature} for angled 2-complexes was introduced by Wise in \cite{Wi}, imposing the condition that planar sections at each vertex have an angle sum $\geq 2\pi$. The definition originates in the so called \textit{Gersten-Pride weight test} (\cite[p.37]{ger}). For this test, we metrize the vertex links of a 2-complex, defining the length of an edge in the link to be the size of the angle at the corresponding corner in the 2-complex. The Gersten-Pride weight test is then satisfied if every cycle in the link of any vertex has length $\geq 2 \pi$.

If angles are assumed to be nonnegative, an angled 2-complex has nonpositive planar sectional curvature if and only if it satisifies the weight test \cite[p.4]{Wis03}. We will here assume that angles are positive, and refer to the angled 2-complexes with nonpositive planar sectional curvature as \textit{conformally nonpositively curved 2-complexes}.

The property of being conformally nonpositively curved turns out to be less restrictive than the property of having a nonpositively curved metric. The presentation complexes of Baumslag Solitar groups are examples of 2-complexes which can be conformally nonpositively curved without admitting a nonpositively curved metric \cite[p.15]{cam}.

In this paper, we contribute to the classification of finite conformally nonpositively curved 2-complexes according to their fundamental groups. 

We will denote a \textit{free face} of a 2-complex to be a cell $c$ of dimension $n=0$ or $1$ with the following properties: The cell $c$ is adjacent to one cell of dimension $n+1$, and not to any other cells of dimension $\geq n+1$. Moreover, the attaching map of the adjacent $(n+1)$-cell is a homeomorphism on the preimage of the interior of $c$.

An \textit{elementary collapse} of a 2-complex is defined to be the removal of the interiors of a free face of dimension $n$ and its adjacent $(n+1)$-cell. A \textit{collapse} of a 2-complex will here denote a sequence of elementary collapses. A 2-complex is \textit{collapsible} if it can be collapsed to a single point.

The first theorem proved in this paper, stated below, classifies the conformally nonpositively curved 2-complexes whose fundamental groups are trivial.

\begin{theorem}\label{1thm}
   A finite, simply connected, conformally nonpositively curved 2-complex is collapsible.
\end{theorem}

Our second theorem, stated below, concerns 2-complexes whose fundamental groups are $\mathbb{Z}$. We will refer to a 2-complex as being \textit{conformally negatively curved} if it satisfies the Gersten-Pride weight test with cycle lengths being strictly greater than $2\pi$. 

\begin{theorem}\label{2thm}
    Let $X$ be a finite, connected, conformally negatively curved 2-complex.
    If $X$ has fundamental group $\mathbb{Z}$, then $X$ can be collapsed to a 1-dimensional cycle.
\end{theorem}

While Theorem \ref{2thm} gives a classification of finite 2-complexes with fundamental group $\mathbb{Z}$ being conformally negatively curved, we do not know whether there are similar results for complexes which are only assumed to be conformally nonpositively curved.

The structure of the paper is as follows: Section 2 provides a background to concepts and invariants related to 2-complexes. Theorem \ref{1thm} is proved in Section 3, using a notion of straight lines as in \cite[p.450]{Wi} and an application of the Combinatorial Gauss-Bonnet Theorem as in \cite[p.455]{Wi}. In Section $4$, we prove Theorem \ref{2thm} by considering the $2$ ends of the universal cover of a conformally negatively curved 2-complex with fundamental group $\mathbb{Z}$. We prove that there cannot be a straight line whose ends belong to the same end of the universal cover, hence excluding the case where an edge in the complex has $3$ or more adjacencies with 2-cells. The remaining case follows from Theorem \ref{1thm} and the fact that the fundamental group of a compact surface cannot be $\mathbb{Z}$.

\section*{Acknowledgements}
This project was conducted under the Summer Research in Mathematics Programme, and received funding by the Trinity Summer Studentship Scheme. I owe thanks to my advisor Henry Wilton for his expert advice and excellent support.

\section{Background}

In this section, we explain the fundamental concepts which will be used in subsequent sections. Most definitions are based on terminology by Bridson and Haefliger \cite{BH}, and Wise \cite{Wi}.

\begin{definition}
    A \textit{(combinatorial) 2-complex} $X$ is constructed by gluing together cells of dimensions $0$, $1$, or $2$ by maps which are homeomorphisms on interiors of cells. The cells of dimensions $0$, $1$ and $2$ are the \textit{vertices}, \textit{edges} and \textit{faces}, respectively, of the 2-complex.
\end{definition}

For our purposes, we will from now on assume that the 2-complexes are connected.

\begin{definition}
    Define the \textit{link} of a vertex $v$ in a 2-complex $X$, denoted $Lk_X(v)$, to be a multigraph, whose vertices correspond to the edges adjacent to $v$ in X. Two (possibly equal) vertices in $Lk_X(v)$ are adjoined by an edge for each corner of a 2-cell the corresponding two edges in $X$ define.
\end{definition}

\begin{definition}
    An \textit{angled 2-complex} will here denote a 2-complex where the corners of each face are prescribed a positive angle, such that the angles in a face with $n$ corners sum up to $(n-2)\pi$.
\end{definition}

Each connected component of the link of a vertex $v$ will be assigned a metric, as follows: Each edge in $Lk_X(v)$ will be given a length equal to the value of the angle at the vertex $v$ for the corresponding corner of a face in $X$. The length between two vertices in the same component of $Lk_X(v)$ will be given by the minimal sum of the lengths of edges in a path connecting them.

\begin{definition}\label{linkc}
    A vertex $v$ in a 2-complex $X$ is said to satisfy the \textit{link condition} if the length of each cycle in $Lk_X(v)$ is $\geq 2\pi$. 
\end{definition}

\begin{definition}
    An angled 2-complex $X$ is said to have \textit{nonpositive planar sectional curvature} if all vertices of $X$ satisfy the link condition. In that case, we will refer to $X$ as a \textit{conformally nonpositively curved complex}. Moreover, if the inequality in Definition \ref{linkc} is strict for each vertex in $X$, we will refer to $X$ as a \textit{conformally negatively curved complex}.
\end{definition}

Note that Wise's definition of an angled 2-complex in \cite{Wi} is slightly more general, with angles allowed to be chosen arbitrarily or taking negative values. We will only be interested in the case of positive angles here. Furthermore, the more general definitions of angled 2-complexes and nonpositive planar sectional curvature allow the angles of a face with $n$ corners to sum up to a value smaller than $(n-2)\pi$. For our purposes however, by increasing angles if necessary, we can without loss of generality assume that the angle sum of a face with $n$ corners is equal to $(n-2)\pi$.

\begin{definition}
    A \textit{combinatorial map} is a continuous map $\tau\colon X \to Y$ between 2-complexes, whose restriction to the interior of single cell in $X$ is a homeomorphism onto the interior of a cell in $Y$.
\end{definition}

We will now define the notion of a straight line segment on a face $f$ in an angled 2-complex. Note that a face $f$ can be equipped with a combinatorial map from a Euclidean polygon $T$ such that angles are preserved. We will refer to $T$ as a \textit{Euclidean polygon associated to f}.

\begin{definition}\label{hom}
    Equip each face $f_j$ of an angled 2-complex with an associated Euclidean polygon $T_j$ and a combinatorial map $\theta_j\colon T_j \to f_j$. A \textit{line segment} on $f_j$ is defined to be the image of a Euclidean line segment on $T_j$ under the map $\theta_i$.
\end{definition}

The following definitions will be used in \textit{Van Kampen's lemma}, relating nullhomotopic edge loops in $X$ to boundaries of simply connected 2-complexes which are homeomorphic to a subset of the Euclidean plane.

\begin{definition}
    A 2-complex is \textit{planar} if it is homeomorphic to a subset of the Euclidean plane.
\end{definition}

\begin{definition}
    A \textit{singular disc diagram} is a simply connected planar 2-complex.
\end{definition}

\begin{definition}
    An \textit{edge path} in a 2-complex $X$ is a sequence of edges adjoined by vertices. An \textit{edge loop} is an edge path which is a circuit.
\end{definition}

\begin{definition}
 A \textit{Van Kampen diagram} for an edge loop $\gamma$ in a 2-complex $X$ is a singular disc diagram $D$ which admits a combinatorial map $\tau\colon D \to X$ such that image of the boundary circuit of $D$ traces out $\gamma$.
\end{definition}

\begin{definition}
 Let $D$ be a van Kampen diagram for a nullhomotopic loop $\gamma$ in a 2-complex $X$, and $\tau\colon D \to X$ its associated map. Let $f_1$, $f_2$ be a pair of faces of $D$ sharing an edge $e$. Let $p_1$ and $p_2$ be the paths tracing out the boundaries of $f_1$ and $f_2$, respectively, starting at the same vertex of $e$ and first traversing $e$. The pair $f_1$, $f_2$ is said to be a \textit{cancellable pair} if the paths $p_1$ and $p_2$ are mapped to the same path in $X$ under $\tau$. The diagram $D$ is said to be \textit{reduced} if it has no cancellable pairs.
\end{definition}

We now state Van Kampen's lemma (\cite[p.442]{Wi}).

\begin{lemma}[Van Kampen's lemma]
    Each nullhomotopic edge loop in a 2-complex $X$ admits a reduced Van Kampen diagram.
\end{lemma}

After one more definition, we conclude this section with a special case of the \textit{Combinatorial Gauss-Bonnet Theorem} stated in \cite{Wi}.

\begin{definition}
Denote the Euler characteristic of a graph $G$ by $\chi(G)$. Let $X$ be an angled 2-complex. 
For a vertex $v$ of $X$, we let $S(v)$ denote the sum of the angles at $v$. We define the \textit{curvature of v} to be the quantity $\kappa(v) = 2\pi-\pi\cdot\chi(Lk_X(v))-S(v)$.
\end{definition}

\begin{theorem}\label{comb}
Let $V(X)$ denote the vertex set of a planar, finite angled 2-complex $X$. If $X$ is simply connected, we have the equality $\sum_{v\in V(X)}\kappa(v) = 2\pi$.
\end{theorem}

\section{Collapsibility}

In this section we show that the finite, simply connected conformally nonpositively curved complexes are collapsible. The proof will require some additional definitions, which are stated below.

\begin{definition}
    An \textit{immersion} between 2-complexes is a locally injective combinatorial map between the complexes. Let $V(X)$ be the vertex set of a 2-complex $X$. A \textit{near-immersion} between 2-complexes is a combinatorial map $\tau \colon X \to Y$ such that $\tau$ is locally injective on $X \backslash V(X)$.
\end{definition}

\begin{definition}
    An \textit{immersed walk} in a (multi)graph will be a walk where no traversed edge is immediately followed by the same edge traversed in the reverse direction.
\end{definition}

\begin{definition}
    A \textit{segmental subdivision} of an angled 2-complex $X$ will be denoted a subdivision $X'$ of $X$, with the following properties:

    1. The edges of $X'$ are line segments in $X$.
    
    2. For each corner in $X'$, consider the face $f$ of $X$ the corner belongs to. The angle of the corner is defined to be the angle between the corresponding line segments in the Euclidean polygon associated to $f$.
    
    Thus $X'$ naturally has the structure of an angled 2-complex.
    
\end{definition}

We now describe how to obtain the vertex links in a segmental subdivision $X'$ of a 2-complex $X$. If a vertex $v$ comes from a vertex of $X$, the link $Lk_{X'}(v)$ is obtained as follows: We take the link of $v$ in $X$ and add vertices corresponding to the new edges of $X'$ meeting $v$. Note that a subdivided angle in $X$ consists of angles in $X'$ summing up to the value of the original angle. Consequently, the edges $Lk_X(v)$ will be subdivided into new edges with lengths adding up to the length of the initial edge.

For a vertex $v$ of $X'$ lying in the interior of an edge $e$ in $X$, the link $Lk_{X'}(v)$ is constructed as follows: Take two vertices $x$ and $y$ corresponding to the two edges in $X'$ adjacent to $v$ along $e$. Adjoin $x$ and $y$ by an edge for each adjacency of $e$ with a face $f$ in $X$, noting that the same face can be adjacent to $e$ in multiple ways. These edges will be assigned the length $\pi$. Subdivide each of these edges by adding vertices for each edge in $X'$ meeting $v$, and give lengths to the new segments according to the angles of corners in $X'$. Again, the lengths of the smaller segments will add up to the length of the original segment. A cycle in the link will consist of two paths between $x$ and $y$, and will thus have length $2\pi$.

If $v$ is a vertex of $X'$ lying in the interior of a face of $X$, the link will be a cycle of length $2\pi$ subdivided according to the edges of $X'$ meeting $v$.

In particular, if $X$ is conformally nonpositively curved, a segmental subdivision $X'$ will also be conformally nonpositively curved with the inherited angle structure.

We will also need a notion of paths on a 2-complex which consist of concatenated line segments. 

\begin{definition}\label{pathray}
    A \textit{segmental path} on a 2-complex $X$ will here denote an edge path in a segmental subdivision of $X$.
\end{definition}

\begin{definition}
    Consider a segmental path $p$ on $X$ and a corresponding subdivision $X'$ of $X$ in which $p$ is an edge path. Assume that $p$ consists of the ordered sequence of vertices $(v_i)_{i\geq 0}$. Let $e_i$ be the edge between $v_i$ and $v_{i+1}$ for $i\geq0$. For $v_i$ not an endpoint of $p$, the segmental path is said to be \textit{straight at $v_i$} if the points in $Lk_{X'}(v_i)$ corresponding to $e_{i}$ and $e_{i+1}$ lie at a distance $\geq \pi$ apart.

    If the segmental path $p$ is straight at $v_i$ for all $i> 0$ such that $v_i$ is not an endpoint of $p$, we say that $p$ is \textit{straight}.
    
    \end{definition}

The main result (same as Theorem \ref{1thm}) in this section is stated below, followed by a lemma that will be used in the proof.

\begin{theorem}\label{mainthm}
    A finite, simply connected conformally nonpositively curved complex is collapsible.
\end{theorem}

\begin{lemma}\label{immersed}
    Let $X$ be a conformally nonpositively curved complex without free faces, and let $v$ be a vertex of $X$. For each point $x$ on an edge in $Lk_X(v)$, there is a point $y$ in $Lk_X(v)$ lying at a distance $\geq\pi$ away from $x$. 
\end{lemma}

\begin{proof}
    Because $X$ has no free face, $Lk_X(v)$ cannot contain a free face of dimension 0. Hence, there is a shortest immersed walk of positive length starting and ending at $x$. This walk will contain a cycle, which because of the link condition has length $\geq2\pi$. Consider a point $y$ which is at a distance $\geq\pi$ away from the first point on the walk belonging to this cycle, the distance being measured along the cycle. Denote the subwalk from $x$ to $y$ in the immersed walk by $p$ and the subwalk from $y$ to $x$ by $q$. Denote the walks obtained by reversing their directions by $p^{-1}$ and $q^{-1}$ respectively.
    
    Since the walk $p$ is followed by $q$ in an immersed walk, the last visited edges of $p$ and $q^{-1}$ must be distinct. Hence, for each walk between $x$ and $y$, continuing the walk along $p^{-1}$ or $q$ will in at least one of the cases yield an immersed walk. In particular, if there were a walk $r$ of distance $<\pi$ between $x$ and $y$, continuing it along $p^{-1}$ or $q$ we can obtain an immersed walk based at $x$ of shorter distance, a contradiction. Hence, $y$ must lie at a distance $\geq\pi$ away from $x$. 
\end{proof}

\begin{proofof}{Theorem~\ref{mainthm}} 
Let $X$ be a finite, simply connected conformally nonpositively curved complex without free faces. If $X$ is 1-dimensional, by the simple connectivity, $X$ must be a single point, since a tree with $n\geq 2$ vertices would have free faces of dimension $0$.

Otherwise, we can pick a 2-dimensional face $f_1$ of $X$ and create a straight segmental path on $X$ as follows: Pick a starting point $v_0$ on the face and continue along a line segment $l_1$ on the face. As soon as we hit the interior of an edge $e$, the absence of free faces allows us to pick an adjacent face $f_2$, which is either different from $f_1$ or has multiple adjacencies with $e$. 

Denote the point of intersection between $e$ and $l_1$ by $v_1$. Continue the segmental path from this point, along a line segment $l_2$ in $f_2$, such that the segmental path is straight at $v_1$.

Now assume that the line segment $l_1$ hits a vertex $v_1$ instead of the interior of an edge. Consider the point $x$ in $Lk_X(v_1)$ corresponding to the segment $l_1$. By Lemma~\ref{immersed}, there is a point $y$ in $Lk_X(v_1)$ lying at a distance $\geq\pi$ from $x$. This point will in turn be associated to the direction of a line segment on a particular face adjacent to $v_1$. Continue the path from $v_1$ along this new line segment. If this segment goes along an edge, we continue the segment until we hit the next vertex, and repeat the previous step with this new segment.

In each step, having traversed a line segment on a face, whenever we hit a vertex or the interior of an edge we repeat the above steps accordingly. From this procedure, we will obtain a straight segmental path on $X$.

Since $X$ is finite, we will eventually come to a point where the segmental path is self-intersecting, or reaches the same edge $e'$ a second time. As soon as either of these cases occurs, we stop the procedure and denote the terminal vertex of the segmental path by $v_n$. Denote this segmental path by $q$. In the first case, we form a closed path $\gamma$ by following the part of $q$ starting and ending at the point of self-intersection. In the second case, let $i$ be the largest value less than $n$ such that $v_i$ lies on $e'$. We form a closed path $\gamma$ by taking the subpath consting of vertices $v_i,\dots,v_n$, and continuing it along $e'$ to connect the endpoints $v_n$ and $v_i$, see Figure~\ref{loop}.

\begin{figure}[h]
\caption{Forming a loop $\gamma$ by connecting a subpath of $q$ along an edge.}
\label{loop}
\includegraphics[width=15cm]{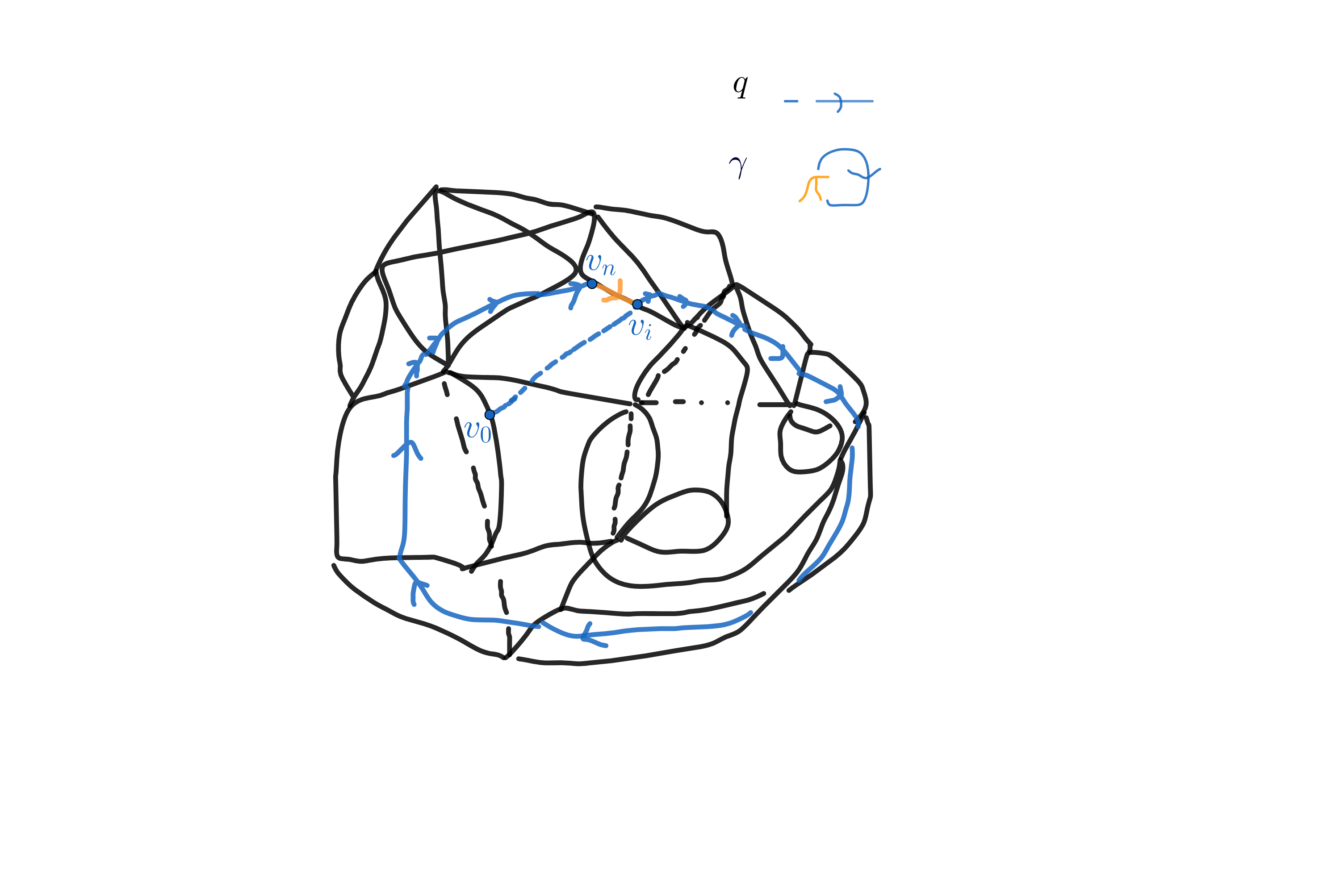}
\centering
\vspace{-100pt}
\end{figure}

Let $X'$ be a segmental subdivision of $X$ such that $\gamma$ is an edge loop in $X'$. The 2-complex $X'$ is simply connected, homeomorphic to $X$ and without free faces. Moreover, since $X'$ is a segmental subdivision, it inherits a conformally nonpositively curved angle structure.

By Van Kampen's lemma, since $\gamma$ is null-homotopic, there exists a reduced Van Kampen diagram $D$ for $\gamma$. Denote by $\tau$ the associated combinatorial map $D\to X'$. Since $\gamma$ is injective by construction, the diagram $D$ is homeomorphic to a disk. Pulling back the angles from $X$ to $D$, we obtain an angle structure on $D$.

When $D$ is reduced, the map $\tau$ is a near-immersion. This means that the link of a vertex $v$ in $D$ locally embeds into the link of the vertex $\tau(v)$ in $X$. In particular, a cycle in $Lk_D(v)$ will map onto a circuit of $Lk_X(\tau(v))$, and will thus have length $\geq 2\pi$. A path between two vertices in $Lk_D(v)$ corresponding to boundary edges of $D$ will be mapped to a walk between vertices corresponding to edges of $\gamma$ in $Lk_X(\tau(v))$. Hence, the interior vertices of $D$ will have an angle sum $\geq 2\pi$. Moreover, since we assumed $q$ to be straight, the angle sums of the boundary vertices of $D$ will all be $\geq \pi$, with possible exceptions for at most two vertices. These exceptional vertices correspond to the endpoints of the subpath of $q$ belonging to $\gamma$.

The links of interior vertices of $D$ are cycles, thus having Euler characteristic $0$. The links of boundary vertices are paths, with Euler characteristic $1$. Hence, the curvatures of the interior vertices of $D$ and of all but at most two vertices on the boundary are $\leq 0$.

If the angles of the exceptional vertices were to have value $0$, the adjacent edges of $\gamma$ would have to overlap, contradicting the injectivity of $\gamma$. Hence, the sum of curvatures of the exceptional vertices must be strictly less than $2\pi$. 

By Theorem \ref{comb}, we obtain the contradictory inequality $$2\pi > \sum_{v\in V(D)}\kappa(v) = 2\pi.$$

In conclusion, each simply connected conformally nonpositively curved 2-complex $X$ must have a free face of dimension $1$ whenever $X$ has 2-dimensional faces. Performing an elementary collapse will result in a complex which is homotopy equivalent to $X$ and with the number of cells decreased by $2$. Moreover, since an elementary collapse does not create any new cycles in the vertex links, the resulting complex will also be conformally nonpositively curved. Thus, we may continue collapsing the complex until the resulting complex is 1-dimensional, and hence is a single point. This shows that $X$ is collapsible.
\end{proofof}

\section{Conformally negatively curved complexes with fundamental group $\mathbb{Z}$}

In this section, we prove a similar result to Theorem \ref{mainthm} for conformally negatively curved complexes with fundamental group $\mathbb{Z}$.

Let $\tilde{X}^{(1)}$ be the 1-skeleton of the universal cover of a 2-complex $X$. We will assign a path metric $d$ to $\tilde{X}^{(1)}$ by assigning all the edges length $1$.

We now state our theorem (same as Theorem \ref{2thm}) which classifies the conformally negatively curved complexes with fundamental group $\mathbb{Z}$.

\begin{theorem}\label{negz}
    A finite, conformally negatively curved complex $X$ with fundamental group $\mathbb{Z}$ can be collapsed to a 1-dimensional cycle.
\end{theorem}

The proof will be divided into 2 parts. First, we consider the case where all edges have at most two adjacencies with 2-cells. Next, we consider the case where there is an edge with three or more adjacencies with 2-cells.\newline

\begin{proofof}{Theorem \ref{negz}}
Assume that $X$ has no free faces.

\textbf{Case 1.} Assume that no edge has three or more adjacenies with 2-cells. Since there are no free faces, all edges of $X$ have $0$ or $2$ adjacencies with 2-cells. Assume that $X$ contains at least one 2-cell. Let $X'$ be the subcomplex of $X$ consisting of the union of the 2-cells of $X$. Since $X$ is obtained from $X'$ by attaching 1-dimensional cells to $X'$ at vertices of $X'$, the fundamental group of $X'$ is either trivial or $\mathbb{Z}$.

Note that $X'$ can be obtained by identifying a finite number of points on a (possibly disconnected) compact surface. By separating these points, we obtain a union of compact, connected surfaces whose fundamental groups must be trivial since compact surfaces cannot have fundamental group $\mathbb{Z}$. Each of these surfaces admits a nonpositively curved structure inherited from $X'$. Theorem \ref{mainthm} tells us that these surfaces each has a free face, which must be 2-dimensional. This means that $X'$, and hence $X$, also has a free face, a contradiction.

Thus, $X$ must be 1-dimensional. We conclude by noting that any 1-dimensional connected complex with fundamental group $\mathbb{Z}$ containing no vertex of degree $1$ must be a cycle.

\textbf{Case 2.} Assume now that there is an edge on $X$ with three or more adjacencies with 2-cells. In this case, we consider three infinite segmental paths starting from a point in the interior of this edge, and choose a different adjacency for each segmental path to traverse when leaving the edge. Moreover, we assume that each segmental path makes an angle $\pi/2$ with the starting edge on the face they first traverse, and that each segmental path is straight. Let $r_1, r_2, r_3$ be the lifts of the respective segmental paths to the universal cover $\tilde{X}$, where the angle structure on $\tilde{X}$ is inherited from $X$.

We will show that any straight segmental path $r$ on $\tilde{X}$ is a proper ray. Indeed, if there is a compact set $K \subset \tilde{X}$ such that $r$ visits $K$ an infinite number of times, the segmental path must either intersect itself or hit the same edge twice in $K$. An argument as in the proof of Theorem \ref{mainthm} shows that we can form a nullhomotopic loop from $r$ which is straight at every point except for at most $2$ points. This will again contradict Theorem \ref{comb}.

Since the fundamental group of $X$ is $\mathbb{Z}$, the universal cover $\tilde{X}$ has $2$ ends. Since $r_1,r_2,r_3$ are proper rays, the pigeonhole principle implies that two of the segmental paths, say $r_1$ and $r_2$, must belong to the same end. Let the vertices of $r_1$ and $r_2$ be $(v_i)_{i\geq0}$ and $(w_i)_{i\geq0}$, respectively, in the order they appear on the segmental paths. We want to show that there is a positive integer $N$ such that, for all $n\in \mathbb{Z}_{\geq 0}$, the vertex $v_n$ can be adjoined by a path to $w_i$ for some $i\geq 0$ along an edge path in $\tilde{X}$ consisting of $\leq N$ edges.

Indeed, by the Schwarz-Milnor lemma, there is a ($\lambda,\varepsilon$)-quasi-isometry $f$ from the 1-skeleton $\tilde{X}^{(1)}$ to $\mathbb{Z}$. By \cite[p.145, Proposition 8.29]{BH}, the ends of $\tilde{X}^{(1)}$ correspond to the ends of $\mathbb{Z}$ under $f$. Assume that the end in $\mathbb{Z}$ corresponding to $r_1$ and $r_2$ is the end corresponding to an increasing sequence of positive integers. Then, for each $m \in \mathbb{Z}$, there is a positive integer $m'$ such that $f(v_n), f(w_n) \geq m$ for all $n\geq m'$. Assume that $M\geq 0$ is such that $f(v_n) \geq f(w_0)$ for all $n\geq M$. Let $L\geq M$ be a positive integer, and let $k = f(v_L)$. 

Let $l$ be the maximum number of edges of a Euclidean polygon associated to a face in $X$. The distance between any two nearby points $w_i, w_{i+1}$ is less than or equal to $l/2$ for $i\geq 0$. Hence, since $f$ is a ($\lambda,\varepsilon$)-quasi-isometry, the distance between any two points $f(w_i), f(w_{i+1})$ is less than or equal to the constant $A = \lambda l/2 + \varepsilon$. Thus, using the facts that $k \geq f(w_0)$ and that $f(w_i) \to \infty$ as $i \to \infty$, there must be a value of $i\geq 0$ such that $|k-f(w_i)|\leq A$.

For this value of $i$, we have $d(v_L,w_i) \leq \lambda(|k-f(w_i)|+\varepsilon)$. In particular, we can find an edge path between $v_L$ and $w_i$ traversing at most $N_0 = \lceil\lambda(A+\varepsilon) \rceil + 1$ edges. Letting $c$ be the maximum value of the distances from $v_n$ to $r_2$ for $0\leq n \leq M$, the value $N = \max\{N_0,\lceil c \rceil \}$ is an upper bound for the distance between any point $v_n$, for some $n\in \mathbb{Z}$, to $r_2$.

Let $L$ be a positive integer, and let $p$ be a path from $v_L$ to $w_i$, for some $i\in \mathbb{Z}_{\geq 0}$, which is contained in an edge path of $\tilde{X}$ consisting of at most $N$ edges. Let $\eta$ be the concatenation of the subpaths $p_1$ and $p_2$ of $r_1$ resp. $r_2$ consisting of the vertices $v_0,\dots,v_L$ and $w_0,\dots,w_i$, respectively. Form a loop $\gamma$ by concatenating $p$ and $\eta$. We consider a minimal subdivision $\tilde{X}'$ of $\tilde{X}$ so that $\gamma$ becomes an edge loop in $\tilde{X}'$.

The loop $\gamma$ is nullhomotopic, and so admits a reduced van Kampen diagram $D$ equipped with a near-immersion $\tau \colon D \to \tilde{X}'$.

Since $X$ is conformally negatively curved, there exists a constant $\varepsilon > 0$ such that all vertices of $X$ have curvature $\leq -\varepsilon$.

Let $B$ denote the number of vertices of $D$ which do not lie on the boundary $\partial D$. The curvatures of vertices on $\partial D$ mapping under $\tau$ to vertices on $\eta$, not being endpoints of this path, will be $\leq 0$. Hence, there are at most $N+1$ vertices on $\partial D$ with positive curvature. Thus, the sum of curvatures of the vertices on $\partial D$ is $\leq (N+1)\pi$. For vertices in the interior of $D$, the curvature is $\leq -\varepsilon$. Theorem \ref{comb} gives us $$2\pi = \sum_{v\in V(D)} \kappa(v) \leq (N+1)\pi -B\varepsilon.$$ Hence, we obtain the bound $B \leq (N-1)\pi/\varepsilon$.

Next, we prove that the value $L$ cannot be arbitrarily large. To do this, will show that the part $Y$ of $\partial D$ mapping to $\eta$ under $\tau$ cannot contain arbitrarily many edges. To begin with, we will find a bound $D'$ on the degree of a vertex of $D\backslash Y$, independent of the choice of $L$. Let $\alpha$ be the value of the smallest angle in the complex $X$. Note that the angles of corners at a vertex in $D\backslash Y$ are inherited from angles of $X$, by the minimality of our subdivision.

If there is a vertex in $D\backslash Y$ of degree $n$, the curvature of this vertex is $\leq 2\pi-n\alpha$. Theorem \ref{comb} gives us 
\begin{equation}\label{ineq}
    2\pi = \sum_{v\in V(D)} \kappa(v) \leq (N+1)\pi + (2\pi-n\alpha) = (N+3)\pi - n\alpha.
\end{equation}

This yields a contradiction for large $n$.

We write $D'$ for the obtained upper bound for the degree of a vertex in $D\backslash Y$, being independent of $L$.

Recall that $l$ denotes the maximal number of edges in a Euclidean polygon associated to a face in $X$. We know that $\eta$ is a straight segmental path, since the angle between the paths $p_1$ and $p_2$ at $v_0=w_0$ is equal to $\pi$. This means, as in the proof of Theorem \ref{mainthm}, that $\eta$ does neither self-intersect, nor cross an edge of $\tilde{X}$ more than once. So $\eta$ cannot visit a face of $\tilde{X}$ more than $l/2$ times. This means that there is an upper bound $E$ on the number of edges of a Euclidean polygon associated to a face in $\tilde{X}'$, and hence on the number of edges of a face in $D$.

Next, we note that each face of $D$ has at least one vertex which does not belong to $Y$. Consequently, the number of faces of $D$ is bounded above by the number of corners of vertices in $D\backslash Y$, and hence by the constant $(B+N-1)D'$. Finally, the number of edges of $D$ is less than or equal to $E$ times the number of faces. In particular, we obtain the bound $$L\leq (B+N-1)D'/E.$$

So $L$ cannot be arbitrarily large, proving that $X$ cannot have an edge with three or more adjacencies with 2-cells. This means that $X$ must have a free face unless $X$ is 1-dimensional, and can thus be collapsed to an 1-dimensional cycle by a similar argument as in the proof of Theorem \ref{mainthm}.

\end{proofof}

\noindent\textsc{Lycka Drakengren},
\textsc{Trinity College},
\textsc{Cambridge CB3 9DH, UK}\newline
\textit{Email address}: \texttt{lmvd2@cam.ac.uk}

\end{document}